\documentclass{amsart}

\usepackage{amssymb, amsmath, enumerate}
\usepackage[colorlinks]{hyperref}

\theoremstyle{definition}
\theoremstyle{plain}
\newtheorem{theorem}{Theorem}[section]

{\Alph{theorem}}
\newtheorem{proposition}[theorem]{Proposition}

\theoremstyle{definition}
\newtheorem*{ack}{Acknowledgments}

\newtheorem{construction}[theorem]{Construction}

\newtheorem{chunk}[theorem]{}

\theoremstyle{remark}

\newtheorem{remark}[theorem]{Remark}
\numberwithin{equation}{theorem}

\numberwithin{equation}{theorem}

\newcommand{\bbz}{\mathbb{Z}}
\newcommand{\col}{\colon}
\newcommand{\dd}{\partial}

\newcommand{\ges}{\geqslant}

\newcommand{\dtensor}[1]{\otimes^{\mathbf L}_{#1}}
\newcommand{\hh}[1]{\operatorname{H}(#1)}
\newcommand{\HH}[2]{\operatorname{H}^{#1}(#2)}
\newcommand{\lch}[2]{{\mathbf{R}\Gamma_{#1}(#2)}}
\newcommand{\Hom}[3]{\operatorname{Hom}_{#1}(#2,#3)}
\newcommand{\rhom}[3]{{\mathbf R}\!\operatorname{Hom}_{#1}(#2,#3)}
\newcommand{\ass}[1][R]{{\operatorname{ass}_{#1}\,}}
\newcommand{\supp}[1][R]{{\operatorname{supp}_{#1}\,}}
\newcommand{\spec}{{\operatorname{Spec}\,}}
\newcommand{\Ker}{{\operatorname{Ker}\,}}

\newcommand{\fa}{\mathfrak{a}}
\newcommand{\fp}{\mathfrak{p}}
\newcommand{\fq}{\mathfrak{q}}
\newcommand{\fm}{\mathfrak{m}}
\newcommand{\fn}{\mathfrak{n}}

\newcommand{\lra}{\longrightarrow}
\newcommand{\xra}{\xrightarrow}

\begin{document}

\title[Support and injective resolutions]{Support and injective resolutions of \\complexes over commutative rings}

\author{Xiao-Wu Chen}
\email{xwchen@mail.ustc.edu.cn}

\address{
Department of Mathematics,
University of Science and Technology of China,
Hefei 230026,
P. R. China.}
%\current{
%Institut f\"ur Mathematik,
%Universit\"at Paderborn,
%33095 Paderborn,
%Deutschland. }

\author{Srikanth B. Iyengar}
\email{iyengar@math.unl.edu}

\address{
Department of Mathematics,
University of Nebraska,
Lincoln NE 68588,
U.S.A.}

\thanks{Chen was supported by Alexander von Humboldt Stiftung and China Postdoctoral Science Foundation. Iyengar was partly supported by NSF grants DMS 0602498 and DMS 0903493.  This research was carried out at the University of Paderborn, where both authors were visiting; they thank the Institute for Mathematics there for hospitality.}


\begin{abstract}
Examples are given to show that the support of a complex of modules over a commutative noetherian ring may not be read off the minimal semi-injective resolution of the complex. The same examples also show that a localization of a semi-injective complex need not be semi-injective.
\end{abstract}

\maketitle

\section{Introduction}
Let $R$ be a commutative noetherian ring, and $\spec R$ the set of prime ideals in $R$. Recall that the support of a finitely generated $R$-module $M$ is the set of points $\fp$ in $\spec R$ such that $M_{\fp}\ne 0$.  For arbitrary modules and, more generally, for complexes of modules, different notions of support have been used. From a homological perspective the one introduced by Foxby in \cite{Fo}, and recalled in Section~\ref{sec:support}, has proved to be quite useful. Foxby~\cite[2.8,2.9]{Fo} proved that a point $\fp$ is in the support of a complex $X$ with $\HH nX=0$ for $n\ll 0$  if and only if the injective hull of $R/\fp$ appears in the minimal semi-injective resolution of $X$.

This note gives examples that show that such a result does not extend to arbitrary complexes, contrary to the claims in \cite[5.1]{KI} and \cite[9.2]{BIK}; see Remark~\ref{rem:false}.

\section{Support and injective resolutions}
\label{sec:support}
For each point $\fp$ in $\spec R$, we write $k(\fp)$ for the residue field $R_{\fp}/\fp R_{\fp}$ of the local ring $R_{\fp}$. The \emph{support} of a complex $X$ of $R$-modules is the subset
\[
\supp X =\{\fp \in \spec R\mid \hh{X\dtensor Rk(\fp)} \neq 0\}\,.
\]
This notion was introduced by Foxby~\cite[p.157]{Fo} under the name `small support', to distinguish it from the `big support', namely, the set $\{\fp\in\spec R\mid {\hh X}_{\fp}\ne0\}$. They coincide when the $R$-module $\hh X$ is finitely generated---see~\cite[2.1]{Fo}---but not in general. Also, $\supp X$
and $\supp \hh X$ need not coincide; see~\cite[9.4]{BIK}.

A point $\fp$ in $\spec R$ is \emph{associated} to an $R$-module $M$ if it is the annihilator of an element in $M$; see \cite[\S6]{Ma}. We write $\ass M$ for the set of associated primes of $M$.

\subsection*{Injective modules}
In what follows $E_{R}(M)$ denotes the injective hull of an $R$-module $M$; see~\cite[\S18]{Ma}. Using \cite[18.4]{Ma}, it is easy to verify that there are equalities
\[
\supp E_{R}(R/\fp) = \{\fp\} = \ass E_{R}(R/\fp)\,.
\]
Let $E$ be an injective $R$-module.
By the structure theorem for injective $R$-modules, see~\cite[18.5]{Ma}, there is an isomorphism
\[
E \cong \bigoplus_{\fp\in\spec R}E(R/\fp)^{\mu(\fp)}\,,
\] 
where each $\mu(\fp)$, which can be $\infty$, depends only on $E$. It follows that one has equalities
\[
\supp E = \{\fp\in \spec R\mid \mu(\fp)\ne 0\} = \ass E\,.
\]
It is this observation that suggests the possibility of reading the support of a complex from its injective resolutions.

\subsection*{Injective resolutions}
We require some basic results concerning injective resolutions; for details see \cite{AFH} and \cite[Appendix B]{Kr}. We say that a complex $I$ of $R$-modules is \emph{homotopically injective} if $\Hom R-I$  preserves quasi-isomorphisms; it is \emph{semi-injective} if in addition each $R$-module $I^{n}$ is injective. For example, a complex $I$ of injective $R$-modules with $I^{n}=0$ for $n\ll0$ is semi-injective. Each complex $X$ of $R$-modules admits a \emph{semi-injective resolution}; that is, a quasi-isomorphism $X\to I$, where $I$ is semi-injective. Moreover, one can choose $I$ so that the extension $\Ker(\dd^{n})\subseteq I^{n}$ is essential for each integer $n$; here $\dd$ is the differential on $I$. Such a \emph{minimal} semi-injective resolution of $X$ is unique, up to isomorphism of complexes. 

\begin{proposition}
\label{prop:support} 
Let $R$ be a commutative noetherian ring and $X$ a complex of $R$-modules. If a complex $I$ of injective modules is quasi-isomorphic to $X$, then
\[
\supp X \subseteq \bigcup_{n\in \bbz} \ass I^n\,.
\]
Equality holds if $I_{\fp}$ is  minimal and homotopically injective for each $\fp\in\spec R$.
\end{proposition}

\begin{remark}
The additional hypotheses on $I$ hold if $R$ is regular, for then any complex of injectives is semi-injective; see \cite[2.4,2.8]{II}. They hold also when $I$ is minimal and $\mathrm{H}^{n}(X)=0$ for $n\ll0$, for then $I^{i}=0$ for $i\ll 0$, so $I$ and its localizations are semi-injective.  Thus, Proposition~\ref{prop:support} extends Foxby's result mentioned earlier.
\end{remark}

\begin{remark}
\label{rem:false} 
In \cite[5.1]{KI} it is claimed  that the inclusion in Proposition~\ref{prop:support} is an equality whenever $I$ is a minimal semi-injective resolution of $X$. This is, however, not the case; see Proposition~\ref{prop:main} for counter-examples. The error in the proof of \cite[5.1]{KI} occurs in the penultimate line, where it is asserted that a certain complex is homotopically injective; what can be salvaged from the argument is Proposition~\ref{prop:support}.  The last line of \cite[9.2]{BIK} is also incorrect. Only conditions (2)--(4) in op.~cit. are equivalent, and are implied by condition (1).
\end{remark}

Proposition~\ref{prop:support} is implicit in \cite[2.1]{FI}, so we provide only a sketch.

Given an ideal $\fa$ in $R$, we write $\Gamma_{\fa}(-)$ for the $\fa$-torsion functor on the category of $R$-modules, and $\lch {\fa}-$ for its right derived functor; see \cite{Fo} or \cite{Li}.

\begin{proof}[Proof of Proposition~\emph{\ref{prop:support}}]
By localization, it suffices to prove the following statement: Let $R$ be a local ring with maximal ideal $\fm$ and residue field $k$. If $\fm$ is in $\supp X$, then the complex $\Gamma_{\fm}(I)$ is non-zero; the converse holds if $I$ is minimal semi-injective. 

It follows from \cite[2.1, 4.1]{FI} that the following conditions are equivalent:
\begin{enumerate}[\quad\rm(i)]
\item $\hh{X\dtensor Rk}\ne 0$;
\item $\hh{\rhom RkX}\ne 0$;
\item $\hh{\lch{\fm}X}\ne 0$.
\end{enumerate}
Since the complex $I$ consists of injective modules and is quasi-isomorphic to $X$, the complexes $\lch{\fm}X$ and $\Gamma_{\fm}(I)$ are quasi-isomorphic; see \cite[3.5.1]{Li}. Therefore, if $\fm$ is in $\supp X$, the complex $\Gamma_{\fm}(I)$ must be non-zero.

Suppose $\fm\not\in\supp X$ holds, so that $\hh{\rhom RkX}=0$. When $I$ is semi-injective there are (quasi-)isomorphisms
\[
\rhom RkX\simeq  \Hom RkI\cong \Hom Rk{\Gamma_{\fm}(I)}\,.
\]
When $I$ is also minimal the differential on $\Hom RkI$ is zero, so $\hh{\Hom RkI}=0$ implies $\Gamma_{\fm}(I)=0$.
\end{proof}

\subsection*{Examples}
Next we focus on our main task; namely, giving examples that show that the inclusion in Proposition~\ref{prop:support} can be strict, even when $I$ is a minimal semi-injective complex. Their construction is motivated by an observation of Neeman~\cite[6.5]{Ne} and recent work of Iacob and Iyengar~\cite[Section 2]{II}. First, we record an elementary remark about associated primes of products.

\begin{remark}
\label{rem:ass}
Let $R$ be a commutative noetherian ring and let $\{M_{\lambda}\}$ be a family of $R$-modules. There are  inclusions
\[
\bigcup_{\lambda} \ass M_{\lambda} \subseteq \ass
{\big(\prod_{\lambda} M_{\lambda}\big)} \subseteq \{\fp\in\spec
R\mid \text{$\fp\subseteq \fq\in\ass M_{\lambda}$ for some
$\lambda$}\}.
\]
Indeed, the inclusion on the left holds since each $M_{\lambda}$ is isomorphic to a submodule of the product. For the one on the right: if a prime $\fp$ is the annihilator of an element $(m_{\lambda})$, then it is contained in the annihilator of each $m_{\lambda}$; pick one that is non-zero.
\end{remark}

In the proof of Proposition~\ref{prop:main} we use the following properties of injective hulls.

\begin{remark}
\label{rem:ihulls}
Let $R$ be a commutative noetherian ring, $\fn$ a prime ideal in $R$, and $E$ the injective hull of $R/\fn$. The following statements hold:
\begin{enumerate}[\quad\rm(1)]
\item Each $r$ in $R\setminus \fn$ is invertible on $E$, hence $E$ has a natural $R_{\fn}$-module structure.
\item
The $R_{\fn}$-module $E$ is Artinian.
\item
As an $R_{\fn}$-module, $E$ has finite length if and only if $\fn$ is a minimal prime.
\end{enumerate}

For (1) see \cite[18.4]{Ma}; for (2), see \cite[18.6]{Ma}; and for (3), see the proof of \cite[18.6(iv)]{Ma}.
\end{remark}

\begin{construction}
Let $R$ be a commutative noetherian ring of Krull dimension at least one; fix a non-minimal prime ideal $\fn$ in $R$. Suppose $R$ contains an element $x$ such that $\{r\in R\mid rx=0\}=(x)$; in particular, $x^{2}=0$.

For example, $R$ could be $\bbz[x]/(x^{2})$, and $\fn = (p,x)$, where $p$ is a prime number.

In what follows we use properties of injective hulls recalled in Remark~\ref{rem:ihulls}. These can be verified directly in the special case when $R=\bbz[x]/(x^{2})$.

Let $E$ be the injective hull of $R/\fn$ over $R$. By the hypothesis on $x$, the complex of $R$-modules $\cdots \xra{x} R\xra{x} R\to 0\to \cdots$, with $0$ in degree $1$, has cohomology only in degree $0$. Thus, applying $\Hom R-E$ to it, one gets a complex of $R$-modules
\[
J = \quad \cdots \lra 0 \lra E \xra{\ x \ } E \xra{\ x\ } E \xra{\ x\ }\cdots
\]
with $0$ in degree $-1$ and $\mathrm{H}^{i}(J)=0$ for $i\ne 0$. Set $M=\mathrm{H}^{0}(J)$; the inclusion $\iota\col M \to J$ is then an injective resolution of $M$ over $R$. It is evidently minimal.

Part (3) of the result below shows that the inclusion in Proposition~\ref{prop:support} can be strict, while (4) shows that a localization of a semi-injective complex need not be homotopically injective. 
We write $\Sigma^{i}X$ for the $i$th suspension of a complex $X$.

\begin{proposition}
\label{prop:main}
Let $X=\prod_{i\in \bbz} \Sigma^iM$ and $I=\prod_{i\in\bbz}\Sigma^{i}J$, viewed as complexes of $R$-modules. The following statements hold.
\begin{enumerate}[{\quad\rm(1)}]
\item The complex $I$ is semi-injective and minimal.
\item
The natural map $\prod_{i\in\bbz}\Sigma^{i}\iota\col  X\to I$ is a quasi-isomorphism.
\item $\supp X = \{\fn\}\subsetneq \ass I^n$, for each  integer $n$.
\item For any prime $\fp$ in $\ass I^n$ with $\fp\ne \fn$, the complex of injective $R_{\fp}$-modules $I_{\fp}$ is acyclic but not contractible, and hence not homotopically injective.
 \end{enumerate}
\end{proposition}

\begin{proof}
Recall that $\iota\col M \to J$ is a quasi-isomorphism.

(1) The complex $\Sigma^{i}J$ consists of injective $R$-modules and $(\Sigma^{i}J)^{n}=0$ for $n<-i$, hence $\Sigma^{i}J$ is semi-injective. Therefore the same holds for $I$, since a product of semi-injective complexes is semi-injective.

As to the minimality, note that the differential $\dd^{n}\col I^{n}\to I^{n+1}$ is the map
\[
\prod_{i\geq n}E \xra{\ \begin{bmatrix} x \\ 0 \end{bmatrix}\ }
\bigg(\prod_{i\geq n}E \bigg)\oplus E = \prod_{i\geq n-1}\!E\,.
\]
Evidently $\Ker(\dd^{n})$ is the submodule $\prod_{i\geq n} M$ of $I^{n}$. It is now straightforward to verify that the extension $\Ker(\dd^{n})\subset I^{n}$ is essential. Thus $I$ is a minimal complex.

(2) holds because a product of quasi-isomorphisms is a quasi-isomorphism.

(3) One has $\supp M = \{\fn\}$. Indeed, $J$ is a minimal injective resolution of $E$ over $R$, so $\supp{M}= \ass {E}= \{\fn\}$. Observe that there is an isomorphism of complexes $X\cong \bigoplus_{i\in \bbz}\Sigma^i{M}$, so $\supp X = \{\fn\}$.

Since the $R$-module $I^{n}$ is isomorphic to $\prod_{i\ges n}E$, Remark~\ref{rem:ass} yields
\[
\{\fn \}=\ass E\subseteq \ass I^{n}\,.
\]
The claim is that this inclusion is strict; equivalently, that there exist elements in $I^{n}=\prod_{i\ges n}E$ that are not $\fn$-torsion.

Indeed, $E$ is the injective hull of $R/\fn$, so it is a module over the local ring $R_{\fn}$. Since $\fn$ is not a minimal prime ideal in $R$, by hypothesis, $R_{\fn}$ does not have finite length, and hence neither does the $R_{\fn}$-module $E$. However, $E$ is Artinian so for each integer $i\ge0$ there must be an element $e_{i}$ in $E$ such that $\fn^{i}\cdot e_{i}\ne 0$. Evidently, the element $(e_{i-n})_{i\ges n}$ in $I^{n}$ is not $\fn$-torsion.

(4) Fix a prime $\fp$ as in the hypothesis. By Remark~\ref{rem:ass}, one has $\fp\subset \fn$ so ${M}_{\fp}=0$, since $M$ is $\fn$-torsion, and hence $X_{\fp}=0$. As $I$ is quasi-isomorphic to $X$, the complex $I_{\fp}$ is quasi-isomorphic to $X_{\fp}$, and hence an acyclic complex of injective $R_{\fp}$-modules. It is also minimal since localization preserves minimality. Since the complex $I_{\fp}$ is non-zero, by the choice of $\fp$, it follows from the minimality that it is not contractible.
\end{proof}
\end{construction}

\begin{ack}
It is a pleasure to thank  Luchezar Avramov for detailed comments and suggestions on earlier versions of this article, and also Lars Winther Christensen and Henning Krause for discussions regarding this work.
\end{ack}


\begin{thebibliography}{99}

\bibitem{AFH}
    L.~L.~Avramov, H.-B. Foxby, S.~Halperin,
    \textit{Differential graded homological algebra}, preprint 2009.

\bibitem{BIK}
D.~J. Benson, S.~B. Iyengar, H.~Krause,
\textit{Local cohomology and support for triangulated categories},
Ann. Sci. \'Ecole Norm. Sup. (4) \textbf{41} (2008), 573--619.


\bibitem{Fo} H.-B.~Foxby,
\textit{Bounded complexes of flat modules},
J. Pure Appl. Algebra \textbf{15} (1979), 149-Ð172.

\bibitem{FI} H.-B.~Foxby, S.~Iyengar,
\textit{Depth and amplitude for unbounded complexes},
  Commutative algebra and its interactions with algebraic geometry
  (Grenoble-Lyon 2001), Contemp. Math. \textbf{331},
  American Math. Soc. Providence, RI, 119--137.

\bibitem{II}
A.~Iacob, S.~B.~Iyengar,
\emph{Homological dimensions and regular rings},
J. Algebra \textbf{322} (2009), 3451--3458.

\bibitem{Kr}
H.~Krause,
\textit{The stable derived category of a noetherian scheme},
Compositio Math. \textbf{141} (2005), 1128--1162.

\bibitem{KI}
H.~Krause,
\textit{Thick subcategories of modules over commutative noetherian rings},
Math. Ann.  \textbf{340} (2008), 733--747.

\bibitem{Li}
J.~Lipman,
\textit{Lectures on local cohomology and duality},
Local cohomology and its applications (Guanajuato, 1999),
Lecture Notes Pure Appl. Math. \textbf{226}, (2002)
Dekker, New York, 39--89.


\bibitem{Ma}
H.~Matsumura, Commutative ring theory, Cambridge Studies in Advanced Math.
\textbf{8}, Cambridge Univ. Press, 1986.

\bibitem{Ne}
A.~Neeman,
\textit{The Grothendieck duality theorem via Bousfield's techniques and Brown representability},
J. Amer. Math. Soc. \textbf{9} (1996), 205--236.



\end{thebibliography}
\end{document}